\documentclass[reqno,12pt]{amsart}
\usepackage[a4paper,margin=2cm]{geometry}
\usepackage[english]{babel}
\usepackage{amsmath,mathrsfs,ragged2e,mathtools,amsaddr,amsfonts,amscd,amssymb,amsthm}
\usepackage{latexsym}
\usepackage{color}
\usepackage{graphicx, graphics}
\usepackage{subfig,float}
\usepackage{tikz,pgfplots}
\definecolor{mycolor1}{rgb}{0.00000,0.44700,0.74100}%
\pgfplotsset{compat=newest,       
	label shift=-.03*\w,
	tick label style={font=\scriptsize},
	label style={font=\tiny},
	legend style={font=\footnotesize},
	every x tick label/.append style={alias=XTick,inner xsep=0pt},
	every x tick scale label/.style={at=(XTick.south east),anchor=base west},
	tick scale binop=\times,
}
\usetikzlibrary{positioning,calc}

\usepackage{comment}
\usepackage[bookmarksopen,colorlinks=true,pdfencoding=auto, psdextra,pagebackref=false,linkcolor=blue,linktocpage=true,citecolor=red]{hyperref}

\allowdisplaybreaks

\numberwithin{equation}{section}
\newtheorem{theorem}{Theorem}[section]

\newtheorem{lemma}{Lemma}[section]

\newtheorem{remark}{Remark}[section]

\newcommand{\bu}{{ u}}
\newcommand{\bv}{{v}}
\newcommand{\bw}{{ w}}

\newcommand{\bvo}{{ v_1}}

\newcommand{\but}{\bu(\cdot,t)}
\newcommand{\bvt}{\bv(\cdot,t)}
\newcommand{\bwt}{\bw(\cdot,t)}

\newcommand{\bus}{\bu(\cdot,s)}
\newcommand{\bvs}{\bv(\cdot,s)}
\newcommand{\bws}{\bw(\cdot,s)}

\newcommand{\gu}{{\nabla \bu}}
\newcommand{\gv}{{\nabla \bv}}
\newcommand{\gvo}{{\nabla \bvo}}
\newcommand{\gw}{{\nabla \bw}}

\newcommand{\gvs}{{\nabla \bvs}}

\newcommand{\lv}{{\Delta \bv}}
\newcommand{\lu}{{\Delta \bu}}
\newcommand{\lw}{{\Delta \bw}}

\newcommand{\mgu}{{|\gu|}}

\newcommand{\mgvo}{{|\gvo|}}

\newcommand{\rsn}{\mathbb{R}^n}
\newcommand{\lis}{{ L^{\infty}}(\Omega)}
\newcommand{\los}{{L^{1}}(\Omega)}
\newcommand{\lts}{{L^{2}}(\Omega)}
\newcommand{\ltps}{{ L^{\frac{2}{p}}}(\Omega)}
\newcommand{\lps}{{ L^{p}}(\Omega)}

\newcommand{\lrs}{{ L^{q_0}}(\Omega)}

\newcommand{\wsq}{{ W^{1,q}}(\Omega)}
\newcommand{\cso}{ {C^{0}(\overline{\Omega})}}

\newcommand{\lros}{{L^{q_0\zeta }(\Omega)}}

\newcommand{\nrv}{\big\|\bv\big\|}
\newcommand{\nrw}{\big\|\bw\big\|}

\newcommand{\nut}{\big\|\but\big\|}
\newcommand{\nvt}{\big\|\bvt\big\|}
\newcommand{\nwt}{\big\|\bwt\big\|}

\newcommand{\ngvt}{\|\nabla\bvt\|}

\newcommand{\ssup}{\sup\limits_{t\in(0,\tmax)}}
\newcommand{\tmax}{T_{\mathrm{max}}}
\newcommand{\tin}{t_0}
\newcommand{\into}{\int_{\Omega}}

\newcommand{\intt}{\int^t_0}
\newcommand{\intti}{\int^t_{\tin}}
\newcommand{\inti}{\int_0^\infty}
\newcommand{\ints}{\int_{\Omega}}
\newcommand{\ds}{\mathrm{d}s}

\newcommand{\dt}{\frac{\mathrm{d}}{\mathrm{d}t}}

\newcommand{\phvo}{\phi(\bvo)}
\newcommand{\povo}{\phi^{'}(\bvo)}
\newcommand{\ptvo}{\phi^{''}(\bvo)}

\newcommand{\bup}{\bu^p}
\newcommand{\upo}{\bu^{p-1}}

\newcommand{\intbo}{\int_{\partial\Omega}}
\newcommand{\absgu}{\lvert \gu \rvert}
\newcommand{\absgv}{\lvert \gv \rvert}
\newcommand{\absgw}{\lvert \gw \rvert}
\allowdisplaybreaks

\title[Tumor-immune cell interactions chemotaxis systems]{Global existence and lower bounds in a class of tumor-immune cell interactions chemotaxis systems}

\author[S. Gnanasekaran, A. Columbu, R. D\'iaz Fuentes, N. Nithyadevi]{$^{\sharp}$Shanmugasundaram Gnanasekaran, $^{\natural}$Alessandro Columbu, \\ $^{\natural}$Rafael D\'iaz Fuentes, ${^\flat}$Nagarajan Nithyadevi}

\begin{document}
	
	\maketitle
	
	\centerline{$^{\natural}$Dipartimento di Matematica e Informatica}
	\centerline{Universit\`{a} degli Studi di Cagliari}
	\centerline{Via Ospedale 72, 09124. Cagliari (Italy)}
	\medskip
	\centerline{${^\sharp}$Department of Mathematics}
	\centerline{Easwari Engineering College}
	\centerline{600089, Chennai (India)}
	\medskip
	\centerline{${^\flat}$Department of Applied Mathematics}
	\centerline{Bharathiar University}
	\centerline{641046, Coimbatore (India)}

	\begin{abstract}
		\justifying
		This paper investigates the properties of classical solutions to a class of chemotaxis systems that model interactions between tumor and immune cells. Our focus is on examining the global existence and explosion of such solutions in bounded domains of  $\mathbb{R}^n$, $n\geq3$, under Neumann boundary conditions. We distinguish between two scenarios: one where all equations are parabolic and another where only one equation is parabolic while the rest are elliptic. Boundedness is demonstrated under smallness assumptions on the initial data in the former scenario, while no such constraints are necessary in the latter. Additionally, we provide estimates for the blow-up time of unbounded solutions in three dimensions, supported by numerical simulations.
		
	\end{abstract}

	
	\section{Introduction and motivations} \label{sec:introduction}
	
	Chemotaxis refers to the movement of microorganisms in response to chemical signals. To gain a deeper understanding of this phenomenon in mathematical terms, one can explore the classical Keller-Segel system. Over the past few decades, researchers have extensively studied the classical Keller-Segel model and its various modified versions, recognizing its significance in the fields of mathematics and biology. For the latest advancements in this area, one can refer to the work by Lankeit and Winkler \cite{lan}. It is noteworthy that Keller and Segel \cite{keller} proposed the consumption model in 1971 to describe the motion of {\it E-coli} bacteria in response to oxygen gradient and consumption. The model is represented by the following system of partial differential equations
	\begin{equation}\label{1.1}
		\begin{cases}
			\bu_t= \Delta \bu-\chi \nabla\cdot(\bu \gv),\\
			\bv_t= \Delta \bv- u\bv.
		\end{cases}
	\end{equation}
	In this system, $\bu$ represents the density of the bacteria, $\bv$ represents the oxygen concentration, and $\chi>0$ is the chemotaxis coefficient. The model predicts the formation of a traveling band of bacteria when the chemotaxis coefficient exceeds a critical value, which can be confirmed through practical observations. Tao \cite{ytao} derived the condition for the global existence of a classical solution to the system \eqref{1.1} for $n\geq 2$. Later a weak global solution to the system \eqref{1.1} is investigated by Tao and Winkler \cite{taowinkler} for $n\geq 3$. They also demonstrated that the solution converges towards constant equilibria as $t\to \infty$. Zhang and Li \cite{zhangli} have improved the condition for the global existence of a classical solution in two spatial dimensions. They showed that the condition stated in \cite{ytao} is needed for $n\geq 3$ and also showed that the convergence of the solution is exponential as time increases. Jiang et al. \cite{jiangwu} studied the blow-up properties of the above system \eqref{1.1} in three spatial dimensions based on the kinetic reformulation technique under the unproved assumption that blow-up occurs. These blow-up results generalized the earlier results available in the literature. In addition, the local non-degeneracy for blow-up points is also indicated. The global existence of the classical solution of the system \eqref{1.1}, under more relaxed assumptions, is proven by Baghaei and Khelghati \cite{kbaghaei}. 
	
	Fuest \cite{mfuest} studied the existence of solution of the following system 
	\begin{equation*}
		\begin{dcases}
			\bu_t= \Delta \bu-\chi \nabla\cdot(\bu \gv),\\
			\bv_t= \Delta \bv- u\bv,\\
			w_t= -\delta w+u.
		\end{dcases}   
	\end{equation*}
	It has been demonstrated that when $n\leq 2$ or $\|v_0\|_{\lis}\leq \frac{1}{3n}$, the system has a global bounded solution. In addition, it is also shown that the solution converges towards equilibrium point as $t\to \infty$. A similar result was subsequently studied by Liu et al. \cite{yliu} with the introduction of a source term $f(u)\leq \mu(u-u^r)$ for $n\geq 2$ in the first equation. Recently both the above two results have been partially improved by Frassu and Viglialoro in \cite{sfrassu} for $n\geq 5$. For more details related to those models, see \cite{jxiang, wzhang}.
	
	Wang et al. \cite{jwang} examined the global existence and asymptotic behaviour of the following predator-prey system with indirect prey--taxis (with positive constants and $f, g, h$ sufficiently regular)
	\begin{equation}\label{1.5}
		\begin{dcases}
			u_t=d_1 \Delta u- \nabla\cdot(\chi(w) u \nabla w)+bug(v)-uh(u),\\
			w_t= d_2\Delta w-\mu w+rv,\\
			v_t= d_3\Delta v+f(v)-ug(v),
		\end{dcases}
	\end{equation}
	providing the convergence rates of the solution. Rather than the prey--taxis, Ahn and Yoon \cite{iahn} analyzed the system, similar to \eqref{1.5} with indirect predator--taxis. In addition, by constructing an appropriate Lyapunov functional, the stability results are also given for $n\leq 2$. Recently with respect to the nonlinear diffusion, the existence of a classical solution was established in \cite{jxing2021} under some parameter assumptions. The authors also discussed in other scenarios the asymptotic behaviour of solutions.

	\quad Chemotaxis models offer valuable insights into immune cell migration mechanisms, enhancing our understanding of immune responses, inflammation, and related diseases. These models are connected to immunology research, drug development, and immunotherapy. Lymphocytes are one of the major immune cells, play a vital role in the immune system. They are produced from bone marrow stem cells and are present in blood and lymph tissues. Lymphocytes work together with other immune cells to protect the body against various microorganisms. They include B cells, which produce antibodies to fight foreign microorganisms, and T cells, which regulate immune responses. Most of the T cells require assistance from another immune cell to be activated. Activated T cells, such as Natural Killer cells (NK cells), specialize in destroying cancer cells and virus-infected cells. The dynamics of the above phenomena are mathematically described by Hu and Tao \cite{hutao} with the help of nonlinear parabolic partial differential equations similar to \eqref{1.5}
	\begin{equation}\label{1.6}
		\begin{dcases}
			\bu_t= \Delta \bu-\chi \nabla\cdot(\bu \gv),\\
			\bv_t= \Delta \bv+\bw-\bv-u\bv,\\
			\bw_t= \Delta \bw- u\bw+\bw(1-\bw).
		\end{dcases}
	\end{equation}
	The authors studied well-posedness, solvability and properties of classical solutions to system \eqref{1.6} in the situation where $\Omega$ is a two-dimensional domain; more precisely, for any $\chi>0$ and any sufficiently regular initial data  solutions are globally bounded. Furthermore, for $\chi$ is small enough and $\frac{1}{|\Omega|}\int_\Omega u_0<1$, convergence to equilibrium points as $t\to\infty$ is proved. On the other hand, also the asymptotic behaviour when $\frac{1}{|\Omega|}\int_\Omega u_0\geq 1$ is discussed.

	\subsection{Aim of the research} \label{subsec:aim}
	Exactly inspired by \cite{hutao}, in this work we aim at investigating model \eqref{1.6} in higher dimensions, in more simplified cases, introducing parameters in the equations to circumvent the nondimensionalization of the model. More precisely, we consider the following initial-boundary value problem that describes the tumor-immune cell interactions given by
	\begin{equation}\label{1}
		\begin{dcases}
			\bu_t= \Delta \bu-\chi \nabla\cdot(\bu \gv),\hspace*{0.5cm} &\text{in} \; \Omega\times(0,\tmax),\\
			\tau \bv_t= \Delta \bv+\alpha\bw-\beta\bv-\gamma u\bv, &\text{in} \; \Omega\times(0,\tmax),\\
			\tau\bw_t= \Delta \bw-\delta u\bw+\mu\bw(1-\bw), &\text{in} \; \Omega\times(0,\tmax),\\
			\frac{\partial \bu}{\partial\nu}=\frac{\partial \bv}{\partial\nu}=\frac{\partial \bw}{\partial\nu}=0, &\text{on} \; \partial\Omega\times(0,\tmax),\\
			\bu(x,0)=\bu_0(x), \quad \tau\bv(x,0)=\tau\bv_0(x), \quad \tau\bw(x,0)=\tau\bw_0(x),&x\in\overline{\Omega}.
		\end{dcases}
	\end{equation}
	Herein $\Omega\subset\mathbb{R}^n$ is a bounded domain, $n\geq 3$, with smooth boundary $\partial\Omega$, $\nu$ denotes the outward unit normal on $\partial\Omega$ and $\tau \in \{0,1\}$. Moreover, with $\tmax$ we are indicating a positive value, possibly equal to infinity, representing the maximum instant of time up to which solutions are defined. The unknown function $\bu=\bu(x,t)$ represents the density of lymphocytes, $\bv=\bv(x,t)$ describes the chemical signal secreted by tumor cells and $\bw=\bw(x,t)$ denotes the density of tumor cells, initially distributed as $u_0(x)$, $\tau v_0(x)$, and $\tau w_0(x)$.  
	Additionally, the parameters $\chi$, $\alpha$, $\beta$, $\gamma$, $\delta$, $\mu$ are assumed to be positive constants. 
	
	The system \eqref{1} describes the interaction of two cells with one chemical secreted by the tumor cells. The term $-\chi\nabla\cdot(\bu\nabla\bv)$ represents the movement of the lymphocytes towards the chemical signal and assume that there is a natural logistic growth term $\mu\bw(1-\bw)$ for the cancer cells with growth coefficient $\mu$. The constant $\alpha$ is the production rate of the chemical by the tumor cells, $\beta$ is the decay rate of the chemical and $\gamma$ is the consumption rate of the chemical by the lymphocytes. The parameter $\delta$ is known as the destruction rate of the tumor cells by the lymphocytes.

	\quad Motivated by the above research works and biological applications which were mentioned earlier, we focus our analysis concerning system \eqref{1} discussing the following scenarios.
	\begin{itemize}
		\item[$\triangleright$] In the fully parabolic case (i.e., $\tau=1$), we prove global existence and boundedness of classical solutions, under smallness conditions involving the parameters and the initial data (see Theorem \ref{t1}). 
		\item[$\triangleright$] In the parabolic-elliptic-elliptic case (i.e., $\tau=0$), we establish the same results as above without any restriction on the parameters and the initial data (see Theorem \ref{t1}).
		\item[$\triangleright$] In 3D domains, lower bounds for the maximal existence time of blowing-up solutions are derived (see Theorem \ref{t3}).
		\item[$\triangleright$] Some numerical simulations in three dimensions are performed (see Section $\S$\ref{sec:numexamples}).
	\end{itemize}

	\subsection{Presentation of the main results} \label{subsec:mainresults}
	In order to formally present the aforementioned analysis, let us suppose that initial values $\bu_0$, $\tau\bv_0$ and $\tau\bw_0$ satisfy
	\begin{equation}\label{2}
		\begin{dcases}
			\bu_0\in\cso,\quad \mbox{with} \quad \bu_0 \geq 0\quad\mbox{in}\: \Omega,\\
			\tau\bv_0, \tau\bw_0\in\wsq,\quad \mbox{for some}\,\, { q} >n,\quad \mbox{with} \quad \tau\bv_0, \tau\bw_0 \geq 0\quad\mbox{in}\: \Omega.
		\end{dcases}
	\end{equation}
	With such a preparation,  the main theorems are stated as follows.
	\begin{theorem}\label{t1}
		For $n\geq 3$, let $\Omega \subset\rsn$ be a  bounded domain with smooth boundary, $q >n$ and $\tau \in \{0,1\}$.   Then  for any initial data $(\bu_0, \tau \bv_0, \tau \bw_0)$ satisfying \eqref{2} and additionally 
		\begin{align}\label{ConditionBoundednessTh}
			\tau \max\left\{\frac{\alpha}{\beta}, \frac{\alpha}{\beta}\|\bw_0\|_{\lis}, \|\bv_0\|_{\lis} \right\}<\frac{\pi}{\chi}\sqrt{\frac{2}{n}},
		\end{align}
		system \eqref{1} admits a unique nonnegative global classical solution $(\bu, \bv,\bw)$ which is uniformly bounded-in-time, in the sense that
		for some $C>0$	
		\begin{align*}
			\nut_{\lis}+\nvt_{\wsq}+\nwt_{\wsq}\leq C, \qquad\qquad \forall \, t > 0.
		\end{align*}
	\end{theorem}
	\begin{remark}
		Conversely to the case $\tau=1$, when  $\tau=0$ we see that \eqref{ConditionBoundednessTh} in Theorem \ref{t1} is automatically fulfilled, independently by the choice of the data of the problem. This stresses how dealing with simplified elliptic versions of chemotaxis models is, in general, more straightforward than for the fully parabolic ones. In this sense, it appears natural to investigate whether the introduction of damping terms (as for instance \textit{logistic sources}; see \cite{tellowinkler}, \cite{winklerhigher}) in the equation for $u$ in problem \eqref{1}, naturally for $\tau=1$, enables to relax restriction \eqref{ConditionBoundednessTh}.
	\end{remark}
	As to the estimate of the blow-up time for solutions to model  system \eqref{1}, we will make use of the following functional:
	\begin{equation}\label{DefFunctional}
		\Psi_\tau= \Psi_\tau(t):= \into u^2 +\tau \into \lvert \nabla v \rvert^4  + \tau \into \lvert \nabla w \rvert^2 \quad \textrm{for all } t \in (0,\tmax),
	\end{equation}
	with initial value
	\begin{equation}\label{defpsitau}
		\Psi_\tau(0) = \into u_0^2 +\tau \into \lvert \nabla v_0 \rvert^4  + \tau \into \lvert \nabla w_0 \rvert^2.
	\end{equation}
	For the upcoming analysis, it is important to note that we must consider an additional restriction on the domain.
	\begin{theorem}\label{t3}
		For $n=3$, $q \geq 4$, and $\tau\in\{0,1\}$, let $\Omega \subset\mathbb{R}^3$ be a bounded convex domain. Moreover, let $(u,v,w)$ be a nonnegative solution to model \eqref{1}, emanating from data as in \eqref{2}, blowing-up at finite time $\tmax$, in the sense that $\limsup_{t \to \tmax} \lVert u(\cdot,t) \rVert_{L^\infty(\Omega)} = +\infty$.
		Then, there exist computable constants $\mathcal{A}_\tau, \mathcal{B}_\tau, \mathcal{C}_\tau$ such that
		\begin{equation} \label{eq:t3}
			T_{\text{max}} \geq \int_{\Psi_\tau(0)}^{+\infty} \frac{d \Psi_\tau(t)}{\mathcal{A}_\tau \Psi_\tau(t)^3 + \mathcal{B}_\tau \Psi_\tau(t)^\frac32 + \mathcal{C}_\tau \Psi_\tau(t)^\tau},
		\end{equation}
		where $\Psi_\tau$ is defined in \eqref{DefFunctional}.
	\end{theorem}

	\quad Our article is organized as follows: In Section $\S$\ref{sec:Preliminary},  we present some basic inequalities, key lemmas and we prove the local existence of classical solutions to system \eqref{1}. Section $\S$\ref{sec:bounds} deals with the boundedness and global existence of such solutions. Section $\S$\ref{sec:blowupTmax} is devoted to blow-up analysis.
	Section $\S$\ref{sec:numexamples} presents the numerical examples of our considered system that validates such blow-up result.

	\section{Preliminaries, local existence and extensibility criterion} \label{sec:Preliminary}
	\quad We recall a few useful inequalities and key lemmas that we are going to use in the next sections.  
	The following proof of the local existence lemma is adopted from well-established arguments proved in \cite{horstman}.
	\begin{lemma}[Local Existence]\label{l1}
		Suppose that $\Omega\subset\rsn$, $n\geq 2,$ is a bounded domain with smooth boundary and $q>n$. Then for each nonnegative initial data satisfying \eqref{2}, there exists $\tmax\in (0,\infty]$ such that  the system \eqref{1} admits a unique nonnegative solution $(\bu, \bv, \bw)$ belonging to
		\begin{align*}
			\bu&\in { C^{0}}\left(\overline{\Omega}\times\left.\left[0,\tmax\right.\right)\right)\cap { C^{2,1}}\left(\overline{\Omega}\times\left(0,\tmax\right)\right),  \\
			\bv, \bw&\in { C^{0}}\left(\overline{\Omega}\times\left.\left[0,\tmax\right.\right)\right)\cap { C^{2,1}}\left(\overline{\Omega}\times\left(0,\tmax\right)\right)\cap { L^{\infty}_{loc}}\left(\left.\left[0,\tmax\right.\right);\wsq\right). 
		\end{align*}
		Furthermore, if $\tmax<\infty$, then
		\begin{equation} \begin{split}
				\limsup_{t\to \tmax}\Big(\nut_{\lis}+\nvt_{\wsq}+\nwt_{\wsq}\Big)= \infty.\label{l1.1}
		\end{split} \end{equation} 
	\end{lemma}
	\begin{proof}
		The proof can be derived by standard arguments involving the Banach fixed point theorem and the parabolic and elliptic regularity theories. (See for instance \cite{horstman}.) 
		Moreover, for $\tau=\{0,1\}$, $\underline{u}\equiv0$ is a subsolution of the equation for $u$ and along to \eqref{2} this implies that $u\geq0$ for all $(x,t)\in\Bar{\Omega}\times(0,\tmax)$, see \cite[$\S$6.4 and $\S$7.1.4]{Evans-2010-PDEs}. An analogous argument holds for $v$ and $w$.  
	\end{proof}
	From now on, with $(u,v,w)$ we will refer to the local-in-time solution to our model \eqref{1} emanating from initial data \eqref{2}, provided by Lemma \ref{l1}.
	\begin{lemma}\label{l2}
		The solution $(u,v,w)$ satisfies
		\begin{equation} \begin{split}
				\nut_{\los}= M_1:=\|\bu_0\|_{\los},\hspace*{3.5cm} \forall t\in(0, \tmax)\label{l2.1}
		\end{split} \end{equation} 
		and
		\begin{align}
			\nwt_{\lis}&\leq M_2:=\max\Big\{1, \tau\|\bw_0\|_{\lis}\Big\}, \hspace*{2cm}& \forall t\in(0, \tmax),\label{l2.2}\\
			\nvt_{\lis}&\leq M_3:=\max \Big \{\frac{\alpha}{\beta}M_2, \tau\|\bv_0\|_{\lis}\Big\}, &\forall t\in(0, \tmax).\label{l2.3}
		\end{align} 
		\begin{proof}
			Integration of the first equation in \eqref{1} gives
			\begin{align*}
				\dt\ints\bu=0 \quad \text{for all $t\in(0,\tmax)$}.
			\end{align*}
			Therefore, this problem preserves mass for $u$, in the sense that \eqref{l2.1} holds. 
			Further \eqref{l2.2} and \eqref{l2.3} are consequences of the elliptic and parabolic maximum principle (we refer again to \cite[$\S$6.4 and $\S$7.1.4]{Evans-2010-PDEs}). 
			Indeed, if $\tau=0$ the third equation makes that $\overline{w}\equiv1$ is a supersolution, so that $0\leq w\leq1$ for all $(x,t)\in\Bar{\Omega}\times(0,\tmax)$. We exploit these information to observe that $\frac \alpha \beta M_2$ is a supersolution for the second equation. The case $\tau=1$ is similar but also the initial condition has to be taken into consideration. 
		\end{proof}
	\end{lemma}
	The following result allows us to establish boundedness of solutions once some uniform-in-time estimates of these in appropriate Lebesgue spaces are derived.
	
	{\em (If not specified, all the constants are positive and the small-letter constants $c_i$ are local to each proof.)}
	\begin{lemma}[Extensibility criterion]\label{l5}
		Suppose there exists $p>\frac{n}{2}\geq 1$ such that
		\begin{align*}
			\ssup \nut_{\lps}<\infty.
		\end{align*}
		Then we have
		\begin{equation*} 
			\ssup\Big(\nut_{\lis}+\nvt_{\wsq}+\nwt_{\wsq}\Big)< \infty. 
		\end{equation*} 
		In particular $\tmax=\infty$ and $u\in L^\infty ((0,\infty),L^\infty(\Omega))$.
	\end{lemma}
	\begin{proof}
		Let $q>n$ and for each fixed $p>\frac{n}{2}$, let us define
		\begin{equation*} 
			\frac{np}{(n-p)_+}=
			\begin{dcases}
				\infty,  &\quad \mbox{if}\quad p\geq n,\\
				\frac{np}{n-p}, & \quad \mbox{if}\quad \frac{n}{2}< p < n,
			\end{dcases}
		\end{equation*} 
		and choose $q<\frac{np}{(n-p)_+}$ and $1<q_0<q$ fulfilling $n<q_0<\frac{np}{(n-p)_+}$
		which enables to take $\zeta>1$, \: $n < \zeta  q_0<\frac{np}{(n-p)_+}$ and $\zeta q_0 \leq q$. We fix arbitrary $t\in(0,\: \tmax)$. Applying the variation of constants formula to the second equation of $\eqref{1}$, we get
		\begin{align*}
			\bvt=e^{-\beta t}e^{t\Delta}\bv_0+\intt e^{-\beta(t-s)} e^{(t-s)\Delta}\Big(\alpha\bws-\gamma u(\cdot, s)\bvs\Big)\ds.
		\end{align*}
		From the above, we have the estimate
		\begin{align*}
			\ngvt_{\lros}\leq&\: e^{-\beta t}\|\nabla e^{t\Delta}\bv_0\|_{\lros}\\
			&+ \intt e^{-\beta(t-s)}\Big\|\nabla e^{(t-s)\Delta}\left(\alpha\bws-\gamma\bus\bvs\right)\Big\|_{\lros}\ds.
		\end{align*}
		By using the estimates for the Neumann heat semigroup  \cite{winkler} and $q_0\zeta \leq q$, we obtain
		\begin{equation*} 
			\begin{split}
				\ngvt_{\lros} \leq&\: c_1e^{-\beta t}\|\bv_0\|_{\wsq} \\
				&+c_2\intt e^{-\beta(t-s)}\left(1+\left(t-s\right)^{-\frac{1}{2}-\frac{n}{2}\left(\frac{1}{p}-\frac{1}{q_0\zeta}\right)}\right)e^{-\lambda(t-s)} \\
				&\times\Big\|\alpha\bws-\gamma \bus\bvs\Big\|_{\lps}\ds \\ 
				\leq&\: c_1\|\bv_0\|_{\wsq} \\
				&+\alpha c_2 M_2\intt e^{-\beta(t-s)}\left(1+\left(t-s\right)^{-\frac{1}{2}-\frac{n}{2}\left(\frac{1}{p}-\frac{1}{q_0\zeta}\right)}\right)e^{-\lambda(t-s)} \\
				&+ \gamma c_2 M_3 \intt e^{-\beta(t-s)}\left(1+\left(t-s\right)^{-\frac{1}{2}-\frac{n}{2}\left(\frac{1}{p}-\frac{1}{q_0\zeta}\right)}\right)e^{-\lambda(t-s)} \\
				&\times\Big\|\bus\Big\|_{\lps}\ds,
			\end{split}
		\end{equation*}
		where $c_1$ and $c_2$ are positive constants, and were we invoked Lemma \ref{l2}. Because of our assumptions $q_0 \zeta<\frac{np}{(n-p)}$ and $q_0 \zeta \leq q$, we can ensure that $\frac{1}{2}+\frac{n}{2}\left(\frac{1}{p}-\frac{1}{q_0\zeta}\right) <1$.
		Since by hypotheses $\lVert u(\cdot,t)\rVert_{L^{p}(\Omega)}<c_3$, for certain $c_3>0$, in view of the fact that the Gamma function gives 
		$\inti e^{-\alpha\psi}\left(1+\psi^{-\frac{1}{2}-\frac{n}{2}(\frac{1}{p}-\frac{1}{q_0\zeta})}\right)e^{-\lambda\psi} < \infty$, we can conclude that
		\begin{equation} 
			\ngvt_{\lros}\leq c_4, \qquad\qquad \forall\, t\in (0, \tmax),\label{l5.6}
		\end{equation} 
		where $c_4>0$. Next let $\tin =\mathrm{max}\{0,\: t-1\}$ and use the variation of constants formula to the first equation of $\eqref{1}$, to get
		\begin{equation} \begin{split}
				u(\cdot, t)=&e^{(t-\tin)\Delta}\bu(\cdot, \tin)-\chi\intti e^{(t-s)\Delta}\nabla\cdot\Big(\bus\gvs\Big)\ds. \label{l5.7}
		\end{split} \end{equation} 
		By taking $\lis$ norm on both sides of \eqref{l5.7}, we obtain
		\begin{align*}
			\nut_{\lis}\leq &\: \big\|e^{(t-\tin)\Delta}\bu(\cdot, \tin)\big\|_{\lis}+\chi\intti\Big\|e^{(t-s)\Delta}\nabla\cdot\Big(\bus\gvs\Big)\Big\|_{\lis}\ds
		\end{align*}
		for all $t\in(0,\tmax)$. If $t\leq 1$, then $\tin=0$ and we can use the maximum principle to get
		\begin{align*}
			\big\|e^{(t-\tin)\Delta}\bu(\cdot,\tin)\big\|_{\lis}= \big\|e^{(t-s_0) \Delta}\bu(\cdot, 0)\big\|_{\lis} \leq \big\|\bu(\cdot, 0)\big\|_{\lis}.
		\end{align*}
		If $t>1$, again using the Neumann heat semigroup properties and Lemma \ref{l2}, with $c_5>0$,
		\begin{align*}
			\big\|e^{(t-\tin)\Delta}\bu(\cdot,\tin)\big\|_{\lis}\leq c_5(t-\tin)^{-\frac{n}{2}}\big\|\bu(\cdot,\tin)\big\|_{\los} \leq c_5 M_1,
		\end{align*}
		because $t-\tin=1$.
		Similarly, for some	$c_6 > 0$ we have
		\begin{equation} \begin{split}
				\nut_{\lis}\leq &\:\max\Big\{\|\bu(\cdot, 0)\|_{\lis}, c_5M_1\Big\}  \\
				&+c_6\intti\left(1+\left(t-s\right)^{-\frac{1}{2}-\frac{n}{2}\left(\frac{1}{q_0}-\frac{1}{\infty}\right)}\right)e^{-\lambda(t-s)}\Big\|\bus\gvs\Big\|_{\lrs}\ds.\label{l5.8}
		\end{split} \end{equation} 
		Here by using the H\"older inequality and the Interpolation inequalities, \eqref{l2.1} and \eqref{l5.6} provide
		\begin{equation} \label{l5.11}
			\begin{split}
				\big\|\bus\gvs\big\|_{\lrs} \leq &\:  \:\big\|\bus\big\|_{{ L^{\widehat{\zeta} q_0}}(\Omega)}\: \big\|\gvs\big\|_{\lros} \\
				\leq &\:  \: \big\|\bus\big\|_{\lis}^{\sigma_2}\: \big\|\bus\big\|_{\los}^{1-\sigma_2}\: \big\|\gvs\big\|_{\lros} \\
				\leq &\:  c_{7} \big\|\bus\big\|_{\lis}^{\sigma_2},
			\end{split}
		\end{equation}
		where $\widehat{\zeta}$ is the dual exponent of $\zeta$ and $\sigma_2=1-\frac{1}{\widehat{\zeta} q_0} \in (0,1)$, for all $s\in (\tin, t)$ and $c_{7}>0$.  Inserting \eqref{l5.11} in \eqref{l5.8}, with the Gamma function we obtain  $\inti\left(1+\psi^{-\frac{1}{2}-\frac{n}{2q_0}}\right)e^{-\lambda\psi} < \infty$, where $\frac{1}{2}+\frac{n}{2q_0}<1$ because of $q_0>n$. This gives
		\begin{equation*} 
			\nut_{\lis}\leq \:\max\Big\{\|\bu(\cdot, 0)\|_{\lis}, c_5M_1\Big\}+c_{8} \big\|\bus\big\|_{\lis}^{\sigma_2}.
		\end{equation*} 
		Finally, using the Young inequality, we obtain
		\begin{equation} \begin{split}
				\nut_{\lis}\leq &c_{9}, \hspace*{2cm} \forall t\in(0, \tmax),\label{l5.13}
		\end{split} \end{equation} 
		where $c_{9}>0$.
		Similarly applying the variation of constants formula to the third equation of $\eqref{1}$, we get
		\begin{align*}
			\|\nabla\bwt\|_{\lros}\leq&\:\|\nabla e^{t\Delta}\bw_0\|_{\lros}+\delta \intt\Big\|\nabla e^{(t-s)\Delta}\bus\bws\Big\|_{\lros}\ds\\
			&+\mu \intt\Big\|\nabla e^{(t-s)\Delta}\bws\big(1-\bws\big)\Big\|_{\lros}\ds,
		\end{align*}
		and we obtain
		\begin{align*}
			\|\nabla\bwt\|_{\lros} \leq&\: c_{10}\|\bw_0\|_{\wsq}  \\
			&+c_{11}\nwt_{\lis}\intt\left(1+\left(t-s\right)^{-\frac{1}{2}-\frac{n}{2}\left(\frac{1}{p}-\frac{1}{q_0\zeta}\right)}\right)e^{-\lambda(t-s)}\Big\|\bus\Big\|_{\lps}\ds  \\
			&+c_{12}\intt\left(1+\left(t-s\right)^{-\frac{1}{2}-\frac{n}{2}\left(\frac{1}{p}-\frac{1}{q_0\zeta}\right)}\right)e^{-\lambda(t-s)}\Big\|\mu\bws(1-\bws)\Big\|_{\lps}.
		\end{align*}	
		In turn, since $w$ is bounded we obtain
		\begin{equation} \begin{split}
				\|\nabla\bwt\|_{\lros}&\leq c_{13}, \hspace*{2cm} t\in(0, \tmax),\label{l7.2}
		\end{split} \end{equation} 
		where $c_{13}>0$. Finally, if by contradiction $\tmax<\infty$, bounds \eqref{l5.6}, \eqref{l5.13} and \eqref{l7.2} would provide an inconsistency to the blow-up criterion \eqref{l1.1}. Hence $\tmax=\infty$, and the proof is given.
	\end{proof}
	This proved result gives a criterion toward boundedness which will be exactly used to esta\-blish such property for solutions to the model we are considering herein. But, the same result is also crucial to ensure that an unbounded solution to such models even blows-up in some energy function. In the specific we have this 
	\begin{remark}\label{fromlinfinitotolpblowup} Let us observe what follows: For $n=3$, if $(u,v,w)$ is a solution to model \eqref{1} which blows-up in finite time $\tmax$ in the sense that 
		\begin{equation*} 
			\limsup_{t\to \tmax} \nut_{\lis} = \infty,
		\end{equation*} 
		then $\lim_{t \to \tmax} \Psi_\tau(t) = +\infty$, where $\Psi_\tau$ is defined in \eqref{DefFunctional}.
		Indeed, if $\into u^2$ were uniformly bounded-in-time on $(0,\tmax)$, necessarily in view of Lemma \ref{l5} we would obtain that $u\in L^\infty((0,\infty);L^\infty(\Omega))$, which is a contradiction. 
	\end{remark}

	\subsection{Some general tools} \label{subsec:GenTools}
	We will rely on the following general results.
	
	The first one (see the main ideas in \cite{qzhang} and \cite{kbaghaei2}, inspired by \cite{ytao}) will be used to prove Theorem \ref{t1} in the case $\tau=1$. 
	\begin{lemma}\label{l6}
		Let $\epsilon\in(0,1)$ and $p>1$. Define the function
		\begin{align*}
			\phi(x):=e^{\zeta(x)}, \quad 0\leq x\leq K,
		\end{align*}
		where
		\begin{align*}
			\zeta(x):=-\frac{l}{2m}x+\frac{\sqrt{4km-l^2}}{2m}\int^x_0\tan\left(\frac{\sqrt{4km-l^2}}{2r}s+\arctan\frac{l}{\sqrt{4km-l^2}}\right)\ds
		\end{align*}
		with $k=(p-1)^2, l=-4(p-1)\epsilon, m=\frac{4}{p}(1+(p-1)\epsilon)$ and $r=\frac{4}{p}(p-1)(1-\epsilon)$. If
		\begin{equation} \begin{split}
				K<\frac{2}{\sqrt{p}}\sqrt{\frac{1-\epsilon}{1+p\epsilon}}\left(\frac{\pi}{2}+\arctan\sqrt{\frac{p}{1+(p-1)\epsilon-p\epsilon^2}}\epsilon\right),\label{l6.1}
		\end{split} \end{equation} 
		then the function $\phi(x)$ is well defined and satisfies the following properties
		\begin{equation} \label{l6.2}
			1\leq \phi(x)\leq \phi(K), \qquad 0\leq \phi^{'}(x)<\infty
		\end{equation} 
		and
		\begin{equation} \begin{split}
				\frac{1}{p}\phi^{''}(x)-\phi^{'}(x)\geq 0.\label{l6.3}
		\end{split} \end{equation} 
		Moreover we have
		\begin{equation} \begin{split}
				|(p-1)\phi(x)-2\phi^{'}(x)|-2\sqrt{(p-1)(1-\epsilon)\phi(x)\left(\frac{1}{p}\phi^{''}(x)-\phi^{'}(x)\right)}=0\label{l6.4}
		\end{split} \end{equation} 
		for all $0\leq x\leq K$.
	\end{lemma}
	\begin{proof}
		First we prove that the function $\phi(x)$ is well defined. For any $p>1$ and $\epsilon\in (0,1)$, we attain
		\begin{align*}
			4km-l^2=\frac{16(p-1)^2}{p}\Big(1+(p-1)\epsilon-p\epsilon^2\Big)
			=\frac{16(p-1)^2}{p}(1-\epsilon)(1+p\epsilon)>0
		\end{align*}
		and
		\begin{equation} \begin{split}
				\frac{\sqrt{4km-l^2}}{2r}s+\arctan\frac{l}{\sqrt{4km-l^2}}
				\geq&\arctan\left(\frac{l}{\sqrt{4km-l^2}}\right)  \\
				\geq&-\arctan\left(\sqrt{\frac{p}{(1-\epsilon)(1+p\epsilon)}}\epsilon\right)
				>-\frac{\pi}{2}\label{l6.5}
		\end{split} \end{equation} 
		due to $4km-l^2>0$. Similarly we arrive at
		\begin{equation} \begin{split}
				\frac{\sqrt{4km-l^2}}{2r}s+\arctan\frac{l}{\sqrt{4km-l^2}}\leq & \frac{\sqrt{p}}{2}\sqrt{\frac{1+p\epsilon}{1-\epsilon}}K-\arctan\left(\sqrt{\frac{p}{(1-\epsilon)(1+p\epsilon)}}\epsilon\right)  \\
				< & \frac{\sqrt{p}}{2}\sqrt{\frac{1+p\epsilon}{1-\epsilon}}  \\
				&\times\left(\frac{2}{\sqrt{p}}\sqrt{\frac{1-\epsilon}{1+p\epsilon}}\left(\frac{\pi}{2}+\arctan\sqrt{\frac{p}{1+(p-1)\epsilon-p\epsilon^2}}\epsilon\right)\right)  \\
				&-\arctan\left(\sqrt{\frac{p}{(1-\epsilon)(1+p\epsilon)}}\epsilon\right)
				< \frac{\pi}{2}. \label{l6.6}
		\end{split} \end{equation} 
		From \eqref{l6.5} and \eqref{l6.6}, we conclude that the function $\zeta(x)$ is bounded. This gives $\phi(x)$ is well defined for $0\leq x\leq K$. Next we attain
		\begin{align*}
			\phi^{'}(x)=&\phi(x)\zeta^{'}(x)\\
			= & \phi(x) \left(-\frac{l}{2m}+\frac{\sqrt{4km-l^2}}{2m}\tan\left(\frac{\sqrt{4km-l^2}}{2r}x+\arctan\frac{l}{\sqrt{4km-l^2}}\right)\right)\\
			\geq& \phi(x)\zeta^{'}(0)
			\geq 0,
		\end{align*}
		for all $0\leq x\leq K$. Since $\phi^{'}(x)$ is continuous and nonnegative, we achieve $1\leq \phi(x)\leq \phi(K)$. Moreover, using \eqref{l6.5}, we have
		\begin{align*}
			\phi^{'}(x)=&\phi(x)\zeta^{'}(x) < \infty,
		\end{align*}
		for all $0\leq x\leq K$; these last relations prove \eqref{l6.2}. Now we calculate
		\begin{equation}\label{l6.7}
			\begin{split}
				\zeta^{''}(x)=&\frac{4km-l^2}{4mr}\sec^2\left(\frac{\sqrt{4km-l^2}}{2r}x+\arctan\frac{l}{\sqrt{4km-l^2}}\right) \\
				=&\frac{4km-l^2}{4mr}\left(1+\tan^2\left(\frac{\sqrt{4km-l^2}}{2r}x+\arctan\frac{l}{\sqrt{4km-l^2}}\right)\right).
			\end{split}
		\end{equation}
		From $\zeta^{'}(x)$, we arrive at
		\begin{equation} \begin{split}
				\frac{4m^2}{4km-l^2}\left(\zeta^{'}(x)+\frac{l}{2m}\right)^2=\tan^2\left(\frac{\sqrt{4km-l^2}}{2r}x+\arctan\frac{l}{\sqrt{4km-l^2}}\right).\label{l6.8}
		\end{split} \end{equation} 
		By substituting \eqref{l6.8} in \eqref{l6.7} we get
		\begin{equation} \begin{split}
				\zeta^{''}(x)=&\frac{4km-l^2}{4mr}\left(1+\frac{4m^2}{4km-l^2}\left(\zeta^{'}(x)+\frac{l}{2m}\right)^2\right)
				=\frac{1}{r}\Big(m\zeta^{'}(x)^2+l\zeta^{'}(x)+k\Big).\label{l6.9}
		\end{split} \end{equation} 
		Next, with the help of \eqref{l6.9}, we calculate
		\begin{align*}
			\frac{1}{p}\phi^{''}(x)-\phi^{'}(x)=&\phi(x)\left(\frac{1}{p}\Big(\zeta^{''}(x)+\zeta^{'}(x)^2\Big)-\zeta^{'}(x)\right)\\
			=&\frac{\phi(x)}{p}\left(\left(\frac{m}{r}+1\right)\zeta^{'}(x)^2+\left(\frac{l}{r}-p\right)\zeta^{'}(x)+\frac{k}{r}\right)\\
			=&\frac{\phi(x)}{p}\left(\frac{p}{(p-1)(1-\epsilon)}\zeta^{'}(x)^2-\frac{p}{1-\epsilon}\zeta^{'}(x)+\frac{p(p-1)}{4(1-\epsilon)}\right)\\
			=&\frac{\phi(x)}{4(p-1)(1-\epsilon)}\left(4\zeta^{'}(x)^2-4(p-1)\zeta^{'}(x)+(p-1)^2\right)\\
			=&\frac{1}{4(p-1)(1-\epsilon)}\frac{1}{\phi(x)}\left((p-1)\phi(x)-2\phi^{'}(x)\right)^2.
		\end{align*}
		This gives \eqref{l6.3} and \eqref{l6.4}.
	\end{proof}
	
	This further lemma is, conversely, employed in the analysis of Theorem \ref{t3}.
	\begin{lemma}
		Let $\Omega$ be a bounded convex domain in $\mathbb{R}^3$. Then for any nonnegative $f \in C^1(\Omega)$ and for every $\varepsilon>0$ it holds
		\begin{equation} \label{ineq:Payne}
			\into f^3 \leq A_1 \left(\into f^2\right)^\frac{3}{2} 
			+ \frac{A_2}{\varepsilon^3} \left(\into f^2\right)^3
			+ A_3 \varepsilon \into \lvert\nabla f\rvert^2,
		\end{equation}
		where 
		\begin{equation*}
			A_1 = \frac{3^{\frac32}}{2 \rho^\frac32},
			\quad
			A_2 = \frac{3^3}{4^\frac{15}{4}}\left(1 + \frac{d}{\rho}\right)^\frac32,
			\quad 
			A_3 = \sqrt{2} \left(1 + \frac{d}{\rho}\right)^\frac32,
		\end{equation*}
		and 
		\begin{equation*}
			\rho := \min_{\partial\Omega}{x \cdot \nu } > 0, \qquad d := \max_{\partial\Omega}{\lvert x \rvert}.
		\end{equation*}
		\begin{proof}
			The proof is a combination of the result in  \cite[Lemma A.2]{lepayne}   and known inequalities. (See details in \cite{VIGLIALORO-JMAA-BlowUp-Attr-Rep}). 
		\end{proof}
	\end{lemma}
	\section{A priori estimates} \label{sec:bounds}
	\subsection{The case \texorpdfstring{$\tau = 1$}{\textepsilon=1}} \label{subsec:tau1}
	In this section we provide the global boundedness of the solution for the system \eqref{1}. The following lemma establishes the $\lps$ estimate of the first component of the solution.
	\begin{lemma}\label{l4}
		For all $p > 1$ and $\epsilon\in(0,1)$, let relation \eqref{l6.1} hold with $K=\chi M_3$, being $M_3$ as in \eqref{l2.3}. Then we have
		\begin{equation*} 
			\hspace*{3cm}\nut_{\lps}\leq C_1, \hspace*{3cm} \forall \, t\in(0,\tmax),
		\end{equation*} 
		for some constant $C_1>0$.
	\end{lemma}
	\begin{proof}
		Set $\bvo=\chi\bv$ to the function $\phi(x)$ defined in Lemma \ref{l6}. Multiplying by $\upo$, $p>1$, the first equation of \eqref{1} and integrating over $\Omega$ we get
		\begin{align*}
			\frac{1}{p}\dt\ints\bup\phvo=&\ints\upo\bu_t\phvo+\frac{1}{p}\ints\bup\povo\bvo_t\\
			=&\ints\upo\phvo\Big(\lu-\nabla\cdot(\bu \gvo)\Big)\\
			&+\frac{1}{p}\ints\bup\povo\Big(\Delta\bvo+\alpha\chi\bw-\beta\bvo-\gamma u\bvo\Big)   \quad \text{for all $t\in(0,\tmax)$}.
		\end{align*}
		Integration by parts leads to
		\begin{align*}
			\frac{1}{p}\dt\ints\bup&\phvo=-(p-1)\ints\bu^{p-2}\phvo\mgu^2-\ints\upo\povo\gu\cdot\gvo \\
			&+(p-1)\ints\upo\phvo\gu\gvo+\ints\bup\povo\mgvo^2-\ints\upo\povo\gu\cdot\gvo \\
			&-\frac{1}{p}\ints\bup\ptvo\mgvo^2
			+\frac{\alpha\chi}{p}\ints\bup\povo\bw-\frac{1}{p}\ints u^p \povo\left(\beta\bvo+\gamma u\bvo\right),
		\end{align*}
		for all $t\in(0,\tmax)$. Combining the terms, one can get on $(0,\tmax)$
		\begin{equation}\label{l4.2} 
			\begin{split}
				\frac{1}{p}\dt\ints\bup\phvo+&(p-1)\epsilon\ints\bu^{p-2}\phvo\mgu^2=-(p-1)(1-\epsilon)\ints\bu^{p-2}\phvo\mgu^2 \\
				&+\ints\left((p-1)\phvo-2 \povo\right)\upo\gu\cdot\gvo +\frac{\alpha\chi}{p}\ints\bup\povo\bw  \\
				&-\ints\left(\frac{1}{p}\ptvo-\povo\right)\bup\mgvo^2 -\frac{1}{p}\ints u^p\povo\Big(\beta\bvo+\gamma u\bvo\Big).
			\end{split} 
		\end{equation} 
		We have by the Gagliardo-Nirenberg and the Young inequalities, thanks to Lemma \ref{l2},
		\begin{equation}
			\begin{split}
				c_1\ints\bup&=c_1\left\|\bu^{\frac{p}{2}}\right\|^2_{\lts}\leq c_2\left(\left\|\nabla\bu^{\frac{p}{2}}\right\|^{2r_0}_{\lts}\,\,\left\|\bu^{\frac{p}{2}}\right\|^{2(1-r_0)}_{\ltps}+\left\|\bu^{\frac{p}{2}}\right\|^2_{\ltps}\right) \\
				&\leq \frac{(p-1)\epsilon}{2}\ints\bu^{p-2}\mgu^2+c_{3} \quad \text{for all $t\in(0,\tmax)$},
			\end{split}
			\label{eq:GNpYuP}
		\end{equation}
		where $r_0=\frac{\frac{p}{2}-\frac{1}{2}}{\frac{p}{2}+\frac{1}{n}-\frac{1}{2}}\in(0,1)$, $c_3$ is a positive constant and some $\epsilon\in(0,1)$.
		
		In view of \eqref{l6.2}, we have $\phvo\geq 1$, so that
		\begin{equation} \label{l4.3} 
			\begin{split}
				c_1\ints\bup\leq \frac{(p-1)\epsilon}{2}\ints\bu^{p-2} \phvo\mgu^2+c_{3}, \quad \text{on $(0,\tmax)$}.
			\end{split} 
		\end{equation} 
		From \eqref{l4.3}, we can estimate on $(0,\tmax)$, again with Lemma \ref{l2}
		\begin{equation}\label{l4.4} \begin{split}
				\frac{\alpha\chi}{p}\ints\bup\povo\bw&\leq \frac{\alpha\chi}{p}\nrw_{\lis}\lVert\phi'\rVert_{L^{\infty}([0,K])}\ints\bup \\
				&\leq \frac{(p-1)\epsilon}{2}\ints\bu^{p-2}\phvo\mgu^2+c_{4}.
		\end{split} \end{equation} 
		By substituting \eqref{l4.4} in \eqref{l4.2} we obtain, for every $t\in(0,\tmax)$,
		\begin{equation} \label{l4.6} \begin{split}
				\frac{1}{p}\dt\ints\bup\phvo+\frac{(p-1)\epsilon}{2}\ints\bu^{p-2}\phvo&\mgu^2\leq \:-(p-1)(1-\epsilon)\ints\bu^{p-2}\phvo\mgu^2 \\
				&+\ints\left((p-1)\phvo-2\povo\right)\upo\gu\cdot\gvo \\
				&-\ints\left(\frac{1}{p}\ptvo-\povo\right)\bup\mgvo^2 +c_4.
		\end{split} \end{equation} 
		Again, using \eqref{l4.3}, we have
		\begin{equation} \begin{split}
				\ints\bup\phvo\leq\phi\left(\chi\nrv_{\lis}\right)\ints\bup\leq \frac{(p-1)\epsilon}{2}\ints\bu^{p-2}\phvo\mgu^2+c_{5}, \quad \text{on $(0,\tmax)$}. \label{l4.7} 
		\end{split} \end{equation} 
		By plugging \eqref{l4.7} in \eqref{l4.6} we can write
		\begin{align*}
			\frac{1}{p}\dt\ints\bup&\phvo+\ints\bup\phvo\leq-(p-1)(1-\epsilon)\ints\bu^{p-2}\phvo\mgu^2\\
			&+\ints\left\lvert(p-1)\phvo-2\povo\right\rvert\upo\mgu\mgvo-\ints\left(\frac{1}{p}\ptvo-\povo\right)\bup\mgvo^2\\
			&+c_6\\
			\leq&\: -\ints\left(\sqrt{(p-1)(1-\epsilon)\phvo}\bu^{\frac{p-2}{2}}\mgu\right)^2
			-\ints\left(\sqrt{\frac{1}{p}\ptvo-\povo}\bu^\frac{p}{2}\mgvo\right)^2\\
			&+\ints\left\lvert(p-1)\phvo-2\povo\right\rvert\upo\mgu\mgvo+c_6\\
			\leq&\: -\ints\left[\left(\sqrt{(p-1)(1-\epsilon)\phvo}\bu^{\frac{p-2}{2}}\mgu\right)^2+\left(\sqrt{\frac{1}{p}\ptvo-\povo}\bu^\frac{p}{2}\mgvo\right)^2\right.\\
			&-\left.2\sqrt{(p-1)(1-\epsilon)\phvo\left(\frac{1}{p}\ptvo-\povo\right)}\bu^{\frac{p-2}{2}+\frac{p}{2}}\mgu\mgvo\right]\\
			&-\ints2\sqrt{(p-1)(1-\epsilon)\phvo\left(\frac{1}{p}\ptvo-\povo\right)}\bu^{\frac{p-2}{2}+\frac{p}{2}}\mgu\mgvo\\
			&+\ints\left\lvert(p-1)\phvo-2\povo\right\rvert\upo\mgu\mgvo+c_6, \quad \text{for all $t\in(0,\tmax)$}.
		\end{align*}
		Collecting the terms in the above inequality, one gets
		\begin{align*}
			\frac{1}{p}\dt\ints\bup&\phvo+\ints\bup\phvo\leq\\ &-\ints\left[\sqrt{(p-1)(1-\epsilon)\phvo}\bu^{\frac{p-2}{2}}\mgu-\sqrt{\frac{1}{p}\ptvo-\povo}\bu^\frac{p}{2}\mgvo\right]^2\\
			&+\ints\left(\left\lvert(p-1)\phvo-2\povo\right\rvert-2\sqrt{(p-1)(1-\epsilon)\phvo\left(\frac{1}{p}\ptvo-\povo\right)}\right)\\
			&\times\bu^{p-1}\mgu\mgvo+c_6  \\
			\leq&\: -\ints\left[\sqrt{(p-1)(1-\epsilon)\phvo}\bu^{\frac{p-2}{2}}\mgu-\sqrt{\frac{1}{p}\ptvo-\povo}\bu^\frac{p}{2}\mgvo\right]^2\\
			&+\ints\varphi\bu^{p-1}\mgu\mgvo+c_6, \quad \text{for every $t\in(0,\tmax)$},
		\end{align*}
		where $\varphi=\left\lvert(p-1)\phvo-2\povo\right\rvert-2 \sqrt{ (p-1)(1-\epsilon)\phvo\left(\frac{1}{p}\ptvo-\povo\right)} = 0$ from \eqref{l6.4}, for all $0\leq \bvo\leq\chi M_3$. As a consequence, we have that for some $c_7>0$
		\begin{equation*}
			\frac{1}{p}\dt\ints\bup\phvo+\ints\bup\phvo\leq c_7 \quad \text{for all} \, t \in (0,\tmax),
		\end{equation*}
		and by ODE arguments, we can conclude that
		\begin{equation*} 
			\ints\bup \leq c_8, \hspace*{2cm} \text{for all } t\in(0, \tmax).
		\end{equation*} 
		This completes the proof.
	\end{proof}

	\subsection{The case \texorpdfstring{$\tau = 0$}{\textepsilon=0}} \label{subsec:tau0}
	In this section we will study problem \eqref{1} in the case were the second and third equations are elliptic, i.e. $\tau=0$. For simplicity we will explicitly make mention to the following model
	\begin{equation}\label{systemtau0}
		\begin{dcases}
			\bu_t= \Delta \bu-\chi \nabla\cdot(\bu \gv),\hspace*{0.5cm} &\text{in} \; \Omega\times(0,\tmax),\\
			0= \Delta \bv+\alpha\bw-\beta\bv-\gamma u\bv, &\text{in} \; \Omega\times(0,\tmax),\\
			0= \Delta \bw-\delta u\bw+\mu\bw(1-\bw), &\text{in} \; \Omega\times(0,\tmax),\\
			\frac{\partial \bu}{\partial\nu}=\frac{\partial \bv}{\partial\nu}=\frac{\partial \bw}{\partial\nu}=0, &\text{on} \; \partial\Omega\times(0,\tmax),\\
			\bu(x,0)=\bu_0(x), &x\in\overline{\Omega}.
		\end{dcases}
	\end{equation}
	In the specific let us establish this result.
	\begin{lemma}\label{Alex1}
		For every $p>1$ there exists $C_2>0$ such that the local solution $(u,v,w)$ of \eqref{systemtau0} satisfies 
		\begin{align*}
			\nut_{\lps}\leq C_2, \qquad \text{for all $t\in(0,\tmax)$}.
		\end{align*}
		\begin{proof}
			We begin with the differentiation of the functional $\Phi(t)=\into u^p$ for which we apply the divergence theorem and obtain
			\begin{equation*}
				\begin{split}
					\Phi'(t)=&p\into \bu^{p-1}\Delta \bu - p\chi\into \bu^{p-1}\nabla\cdot \left( \bu\nabla \bv\right)\\
					=&-p(p-1)\into \bu^{p-2}\left|\nabla\bu \right|^2+p(p-1)\chi\into\bu^{p-1}\nabla\bu\cdot\nabla\bv\\
					=&-\frac{4(p-1)}{p}\into \left|\nabla u^{\frac p2}\right|^2-(p-1)\chi\into u^p\Delta v \quad \text{on $(0,\tmax)$.}
				\end{split}
			\end{equation*}
			We now make use of the second equation in conjunction with the fact that $u,v$ and $w$ are positive and $w$ is bounded as in \eqref{l2.2}, so to write
			\begin{equation}
				\begin{split}
					\Phi'(t)=&-\frac{4(p-1)}{p}\into \left|\nabla u^{\frac p2}\right|^2-(p-1)\chi\into u^p\left(-\alpha w +\beta v +\gamma uv \right)\\\label{phiprimo}
					\leq&-\frac{4(p-1)}{p}\into \left|\nabla u^{\frac p2}\right|^2+\alpha(p-1)\chi M_2\into u^p \quad \text{on $(0,\tmax)$.}
				\end{split}
			\end{equation}
			Exactly as in \eqref{eq:GNpYuP}, we can say that for any $c_1>0$ there is a proper $c_2>0$ such that 
			\begin{equation}\label{gn}
				c_1\into u^p\leq\frac{2(p-1)}{p} \into \lvert\nabla u^{\frac p2} \rvert^2 +c_2 \quad \text{on $(0,\tmax)$}.
			\end{equation}
			In this way \eqref{phiprimo} and \eqref{gn} lead to this initial problem:
			\begin{equation*}
				\begin{dcases}
					\Phi'(t)\leq 2 c_2-c_1\Phi(t) & \text{on $(0,\tmax)$},\\
					\Phi(0)=\into u_0^p.
				\end{dcases}
			\end{equation*}
			Naturally, the above problem is solvable and $\Phi(t)\leq \max{ \left\{\into u_0^p, \frac{2c_2}{c_1}\right\}}$.
		\end{proof}
	\end{lemma}
	\subsubsection*{Proof of Theorem \ref{t1}.}
	For the case $\tau = 0$, we directly consider Lemmas \ref{Alex1} and \ref{l5}. For $\tau = 1$, with the aim of exploiting Lemma \ref{l4} (and successively again Lemma \ref{l5}), it is sufficient to 
	verify \eqref{l6.1} with $K:= \chi M_3<\sqrt{\frac{2}{n}}\pi$. We can choose
	\begin{align*}
		\epsilon=\frac{\pi^2-\frac{n}{2}K^2}{2\left(\pi^2+(\frac{n}{2})^2K^2\right)}\in (0, 1)
	\end{align*}
	such that
	\begin{align*}
		K=\frac{2}{\sqrt{\frac{n}{2}}}\sqrt{\frac{1-2\epsilon}{1+2\epsilon\frac{n}{2}}}\frac{\pi}{2}<\frac{2}{\sqrt{\frac{n}{2}}}\sqrt{\frac{1-\epsilon}{1+\epsilon\frac{n}{2}}}\frac{\pi}{2},
	\end{align*}
	for some $p>\frac{n}{2}$. Then we have
	\begin{align*}
		K<\frac{2}{\sqrt{p}}\sqrt{\frac{1-\epsilon}{1+p\epsilon}}\frac{\pi}{2}<\frac{2}{\sqrt{p}}\sqrt{\frac{1-\epsilon}{1+p\epsilon}}\left(\frac{\pi}{2}+\arctan\sqrt{\frac{p}{1+(p-1)\epsilon-p\epsilon^2}}\epsilon\right). \hfill \qed 
	\end{align*}

	\section{Lower bounds for the maximal existence time of solutions in \texorpdfstring{$\mathbb{R}^{3}$}{R3}}
	\label{sec:blowupTmax}
	This section provides the lower bounds for the maximal existence time of solutions in $\mathbb{R}^3$. In this direction, we have to rely on  what follows.
	\begin{lemma}
		For $n=3$ and $\tau\in\{0,1\}$, let $\Omega \subset\mathbb{R}^3$ be a bounded convex domain and let $(u,v,w)$ be a solution to model \eqref{1}  blowing-up at finite time $\tmax$, in the sense that $\limsup_{t \to \tmax} \lVert u(\cdot,t) \rVert_{L^\infty(\Omega)} = +\infty$.
		Then, there exist computable constants $\mathcal{A}_\tau, \mathcal{B}_\tau, \mathcal{C}_\tau$ such that
		\begin{equation} \label{eq:t3demo}
			\Psi'_\tau(t) \leq \mathcal{A}_\tau \Psi_\tau(t)^3 + \mathcal{B}_\tau \Psi_\tau(t)^\frac32 + \mathcal{C}_\tau \Psi_\tau(t)^\tau \quad \text{on} \; (0,\tmax),
		\end{equation}
		where $\Psi_\tau$ is defined in \eqref{DefFunctional}.
		
		\begin{proof}
			We will distinguish the fully parabolic case, i.e., $\tau=1$, and the parabolic-elliptic-elliptic case, i.e., $\tau=0$.
			\subsection*{The case \texorpdfstring{$\tau = 1$}{\textepsilon=1}}  
			We have to deal with the following  energy function 
			\begin{equation} \label{eq:blowupFunc}
				\Psi_1(t) := \into u^2 + \into \lvert \nabla v \rvert^4  + \into \lvert \nabla w \rvert^2 \qquad \text{for all}\; t\in(0, \tmax). 
			\end{equation}
			Let us denote the integrals in \eqref{eq:blowupFunc} by $I_1(t)$, $I_2(t)$, and $I_3(t)$ respectively. Let us then compute their derivatives one at a time.
			
			For the first integral $I_1(t)$ we have thanks to the Young inequality
			\begin{equation}\label{eq:1Blow-upEllittico}
				\begin{split}
					\frac{d I_1(t)}{d t} &= 2\into u \left[ \Delta \bu-\chi \nabla\cdot(\bu \gv) \right] = - 2\into \lvert \gu \rvert^2 + 2\chi \into \gu \cdot u\gv  \\
					&\leq -\frac32 \into \absgu^2 + 2\chi^2 \into u^2 \absgv^2 \quad 
					\text{on} \; (0,\tmax).
				\end{split}
			\end{equation}
			
			The Young inequality applied again to $u^2 \absgv^2$ and successively  the inequality in \eqref{ineq:Payne}, with $\varepsilon = \frac{1}{4 A_3 \chi^2}$ for $\into u^3$ and  $\varepsilon = \frac{27}{16 A_3 \chi^2}$ for $\into \absgv^6$, provide for all $t\in (0,\tmax)$
			\begin{equation} \label{eq:dI1}
				\begin{split}
					\frac{d I_1(t)}{d t} &\leq -\frac{3}{2} \into \absgu^2 + 2\chi^2 \into u^3 +  \frac{8\chi^2}{27} \into \absgv^6 \\
					&\leq -\into \absgu^2 + 2A_1 \chi^2 \left(\into u^2\right)^\frac{3}{2} 
					+ 2^7 A_2 A_3^3 \chi^8 \left(\into u^2\right)^3 \\
					&\quad + \frac{8A_1\chi^2}{27} \left(\into \absgv^4\right)^\frac{3}{2} 
					+ \frac{2^{15}A_2 A_3^3\chi^8}{3^{12}} \left(\into \absgv^4\right)^3
					+ 2 \into \absgv^2 \lvert D^2 v\rvert^2.
				\end{split}
			\end{equation} 
			In the last step we invoked as well the inequality $\lvert\nabla \absgv^2\rvert^2 \leq 4 \absgv^2 \lvert D^2 v\rvert^2$, applied to the integral term $\frac{1}{2}\int_\Omega |\nabla|\nabla  v|^2|^2$ .
			
			With the second term $I_2(t)$ we obtain, once again by relying on the divergence theorem and moreover by virtue of the identity $\Delta (\absgv^2) = 2\gv \cdot \nabla\lv +2 \lvert D^2 v\rvert^2$, for $t\in(0,\tmax)$:
			\begin{equation*}
				\begin{split}
					\frac{d I_2(t)}{d t} &= 4 \into \absgv^2 \gv \cdot \left[ \nabla \lv  + \alpha \gw - \beta \gv - \gamma v\gu - \gamma u\gv \right] \\
					&=  4 \into \absgv^2 \gv \cdot \nabla\lv + 4\alpha \into \absgv^2 \gv \cdot \gw - 4 \beta \into \absgv^4 \\ 
					&\quad - 4 \gamma \into \absgv^2 v\gv \cdot \gu- 4 \gamma \into \absgv^4 u \\
					&= 2 \into \absgv^2 \Delta (\absgv^2) - 4 \into \absgv^2 \lvert D^2 v\rvert^2 + 4\alpha \into \absgv^2 \gv \cdot \gw  - 4 \beta \into \absgv^4\\ 
					&\quad - 4 \gamma \into \absgv^2 v\gv \cdot \gu - 4 \gamma \into \absgv^4 u\\
					&= 2 \intbo \absgv^2 \nabla (\absgv^2)\cdot \nu  - 2\into \lvert \nabla \absgv^2 \rvert^2 - 4 \into \absgv^2 \lvert D^2 v\rvert^2 \\ 
					&\quad + 4\alpha \into \absgv^2 \gv \cdot \gw - 4 \beta \into \absgv^4 - 4 \gamma \into \absgv^2 v\gv \cdot \gu - 4 \gamma \into \absgv^4 u.
				\end{split}
			\end{equation*}
			In turn, the inequality $\nabla (\absgv^2)\cdot \nu \leq 0$ on $\partial \Omega$ proved in \cite{TaoWinkler}, after ignoring nonpositive terms, leads to
			\begin{equation*}
				\begin{split}
					\frac{d I_2(t)}{d t} &\leq - 2\into \lvert \nabla \absgv^2 \rvert^2 - 4 \into \absgv^2 \lvert D^2 v\rvert^2  + 4\alpha \into \absgv^2 \gv \cdot \gw\\ 
					&\quad - 4 \gamma \into \absgv^2 v\gv \cdot \gu \\
					&\leq - 2\into \lvert \nabla \absgv^2 \rvert^2 - 4 \into \absgv^2 \lvert D^2 v\rvert^2 + 4\left(\alpha + 4(\gamma M_3)^2  \right) \into \absgv^6 \\
					&\quad + \alpha \into \absgw^2 + \frac{1}{4} \into \absgu^2,
					\quad \text{with} \; t \in (0,\tmax)
				\end{split}
			\end{equation*}
			having exploited the Young inequalities, as well as $\lVert v \rVert_{L^\infty(\Omega)}\leq M_3$ (recall \eqref{l2.3}).
			
			The inequality in \eqref{ineq:Payne}, with the choice $\varepsilon = \frac{5}{8 A_3 \left(\alpha + 4 (\gamma M_3)^2\right)}$, conduces to
			\begin{equation} \label{eq:dI2}
				\begin{split}
					\frac{d I_2(t)}{d t} &\leq  - 4 \into \absgv^2 \lvert D^2 v\rvert^2 
					+ \alpha \into \absgw^2 + \frac{1}{4} \into \absgu^2 + \frac12 \into \lvert\nabla \absgv^2\rvert^2 \\
					&\quad  + 4\left(\alpha + 4(\gamma M_3)^2 \right) A_1 \left(\into \absgv^4\right)^\frac{3}{2}  + \frac{2^{11} A_2 A_3^3 \left(\alpha + 4(\gamma M_3)^2 \right)^4}{5^3} \left(\into \absgv^4\right)^3
					\\
					&\leq - 2 \into \absgv^2 \lvert D^2 v\rvert^2 
					+ \alpha \into \absgw^2 + \frac{1}{4} \into \absgu^2\\
					&\quad  + 4\left(\alpha + 4(\gamma M_3)^2 \right) A_1 \left(\into \absgv^4\right)^\frac{3}{2}  + \frac{2^{11} A_2 A_3^3 \left(\alpha + 4(\gamma M_3)^2 \right)^4}{5^3} \left(\into \absgv^4\right)^3
				\end{split}
			\end{equation}
			on $(0,\tmax)$.
			
			\medskip
			With analogous steps, last integral $I_3(t)$ leads, for $t\in (0,\tmax)$, to 
			\begin{equation} \label{eq:dI3}
				\begin{split}
					\frac{d I_3(t)}{d t} &= 2 \into \gw \cdot \left[ \nabla \lw - \delta w\gu - \delta u\gw +  \mu \gw - 2 \mu w\gw \right] \\
					&= 2 \into \gw \cdot \nabla \lw - 2\delta \into w\gw \cdot \gu - 2\delta \into u \absgw^2 + 2 \mu \into \absgw^2 \\
					&\quad - 4 \mu \into w\absgw^2 \\
					&= \into \Delta (\absgw^2)  - 2\into \lvert D^2 w \rvert^2 - 2\delta \into w\gw \cdot \gu - 2\delta \into u \absgw^2 \\
					&\quad + 2 \mu \into \absgw^2 - 4 \mu \into w\absgw^2 \\
					&\leq - 2\delta \into w\gw \cdot \gu + 2 \mu \into \absgw^2 
					\leq 2\delta M_2 \into \absgw \absgu + 2 \mu \into \absgw^2 \\
					&\leq 2\left(2 (\delta M_2)^2 +  \mu\right) \into \absgw^2 + \frac{1}{4} \into \absgu^2,
				\end{split}
			\end{equation}
			recalling that $\lVert w \rVert_{L^\infty(\Omega)}\leq M_2$ and after having used again the Young inequality.
			
			By combining \eqref{eq:dI1}, \eqref{eq:dI2}, and \eqref{eq:dI3}, we can conclude for $t \in (0,\tmax)$
			\begin{equation*} 
				\begin{split}
					\Psi_1'(t) &\leq \left(\alpha + 2\left(2 (\delta M_2)^2 +  \mu\right)\right) \into \absgw^2 
					+ 2A_1 \chi^2 \left(\into u^2\right)^\frac{3}{2} \\
					&\quad + 4A_1 \left(\frac{2\chi^2}{27} + \left(\alpha + 4(\gamma M_3)^2 \right) \right) \left(\into \absgv^4\right)^\frac{3}{2} 
					+ 2^7 A_2 A_3^3 \chi^8 \left(\into u^2\right)^3 \\
					&\quad + 2^{11} A_2 A_3^3 \left(\frac{2^{4} \chi^8}{3^{12}} + \frac{\left(\alpha + 4(\gamma M_3)^2 \right)^4}{5^3}\right)  \left(\into \absgv^4\right)^3  \\
					&\leq \mathcal{A}_1 \Psi_1(t)^3 + \mathcal{B}_1 \Psi_1(t)^\frac32 + \mathcal{C}_1 \Psi_1(t),
				\end{split}
			\end{equation*}
			with 
			\begin{equation*} 
				\begin{split}
					\mathcal{A}_1 &:= 2^7 A_2 A_3^3\chi^8 \max \left\{ 1 \; , \; \frac{2^{8}}{3^{12}} + \frac{2^{4}  \left(\alpha + 4(\gamma M_3)^2 \right)^4}{5^3 \chi^8} \right\},  \\
					\mathcal{B}_1 &:=  2A_1 \chi^2 \max \left\{1 \; , \; \frac{4}{27} + \frac{2\left(\alpha + 4(\gamma M_3)^2 \right)}{\chi^2} \right\}, \\	
					\mathcal{C}_1 &:= \alpha + 4 (\delta M_2)^2 + 2 \mu.
				\end{split}
			\end{equation*}
			
			\subsection*{The case \texorpdfstring{$\tau = 0$}{\textepsilon=0}} 
			Naturally, this simplification makes that of the three expressions $I_1$, $I_2$ and $I_3$ in \eqref{eq:blowupFunc}  only the first one takes part in the analysis, and in particular 
			the energy function is reduced into
			\begin{equation*} 
				\Psi_0(t) := \into u^2\quad \text{for all} \; t \in (0,\tmax).
			\end{equation*}
			Subsequently,  from  \eqref{eq:1Blow-upEllittico} we can write, by applying twice the divergence theorem,   
			\begin{equation*} 
				\frac{d I_1(t)}{d t} = 2\into u \left[ \Delta \bu-\chi \nabla\cdot(\bu \gv) \right] = - 2\into \lvert \gu \rvert^2 -\chi  \into u^2 \Delta v,  \quad \textrm{on} \quad (0,\tmax)
			\end{equation*}
			which thanks to the second equation in system \eqref{1} turns into
			\begin{equation*} 
				\begin{split}
					\Psi_0'(t) &= -2 \into \absgu^2 + \chi \alpha \into u^2 w - \chi \beta \into u^2 v - \chi \gamma \into u^3 v \\
					&\leq -2 \into \absgu^2 + \chi \alpha \into u^2 w \\
					&\leq -2 \into \absgu^2 + \into u^3 + \frac{4}{27 \chi \alpha} M_2^3 \lvert \Omega \rvert \quad  \textrm{for all } \; t \in (0,\tmax);
				\end{split}
			\end{equation*}
			we observe that in the last step we used the Young inequality and the estimate on $w$ given in \eqref{l2.2}.
			
			Since the case $\tau=1$ has been deeply discussed,  we understand that we may omit for this situation some  details. In particular, by relying on the inequality \eqref{ineq:Payne}, the integral $\int_\Omega u^3$ can be properly estimated in terms of $\int_\Omega \lvert\nabla u \rvert^2$, $\left(\into u^2\right)^3$ and $\left(\into u^2\right)^\frac{3}{2}$. Eventually, we obtain for some $\mathcal{A}_0$, $\mathcal{B}_0$ and $\mathcal{C}_0$ positive this inequality 
			\begin{equation*} 
				\Psi_0'(t) \leq \mathcal{A}_0 \Psi_0(t)^3  + \mathcal{B}_0 \Psi_0(t)^\frac32 + \mathcal{C}_0 \quad \textrm{on} \quad (0,\tmax). \qedhere
			\end{equation*} 	
		\end{proof}
	\end{lemma}
	
	\subsubsection*{Proof of Theorem \ref{t3}.}
	Integrating \eqref{eq:t3demo} in $(0,\tmax)$, and taking into account that from Remark \ref{fromlinfinitotolpblowup} one has that $\lim_{t\to\tmax}{\Psi_\tau (t)}=+\infty$ we obtain \eqref{eq:t3} which completes the proof.
	\begin{remark}
		We understand, that a similar analysis derive in Theorem \ref{t3} can be carried out in domains included in $\mathbb{R}^n$ with $n\geq 4$. In this case, \eqref{ineq:Payne} cannot be employed and an alternative inequality is required. In this direction, one can invoke some appropriate Sobolev embedding, exactly in the spirit of \cite{AndDeng}.
	\end{remark}
	
	\section{A numerical simulation in a cube} \label{sec:numexamples}
	Far from wishing present a strong numerical analysis, in this section we dedicate to some 3-dimensional simulations. Indeed, if from the one hand, in two dimensions, we have recalled that no blow-up occurs, from the other, explosion at finite time in higher dimensions is an open problem. Despite that, some evidences may suggest blow-up in a cube; we understand that at least in a first level these are sufficient to motivate the estimate of the blow-up time derived in Theorem \ref{t3}.
	
	\quad The numerical simulation is conducted using the Finite Difference Method, drawing inspiration from the works \cite{ChertokKur, Chertoketal}. Specifically, we employ first-order central difference for the first derivatives, second-order central difference for the diffusion term (Crank--Nicolson scheme), and the Lax--Friedrichs scheme for the chemotactic term.
	
	\begin{figure}[h]
		\centering
		\subfloat[Lymphocyte: $u$ at time $t=5\times10^{-6}$.]{\includegraphics[width=5.5cm, height=5cm]{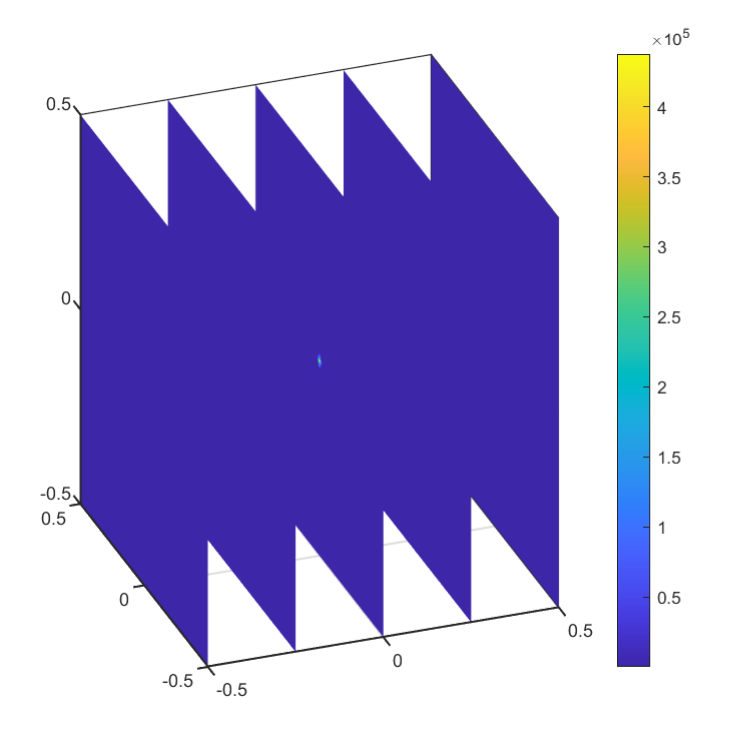}}
		\subfloat[Chemical concentration: $v$ at time $t=5\times10^{-6}$.]{\includegraphics[width=5.5cm, height=5cm]{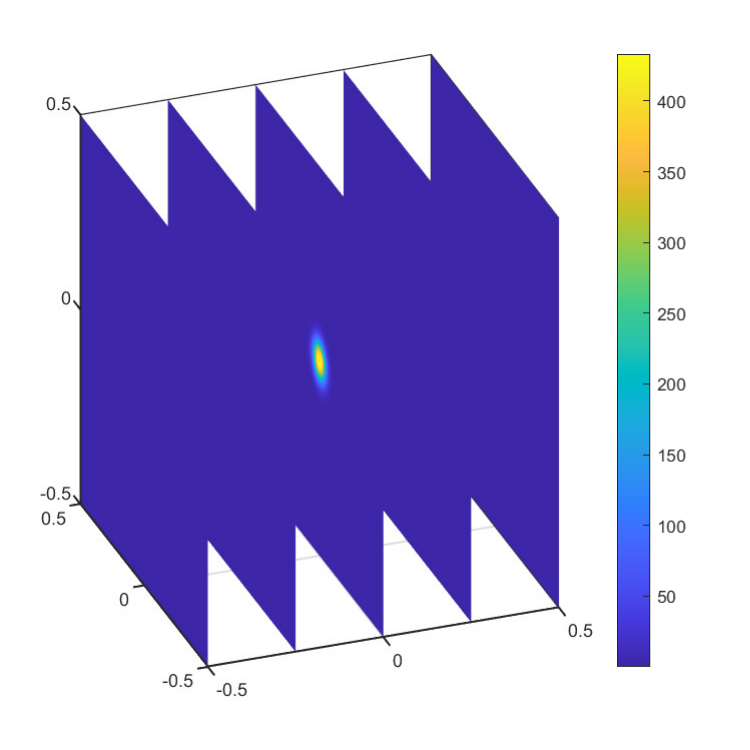}}
		\subfloat[Tumor cell: $w$ at time $t=5\times10^{-6}$.]{\includegraphics[width=5.5cm, height=5cm]{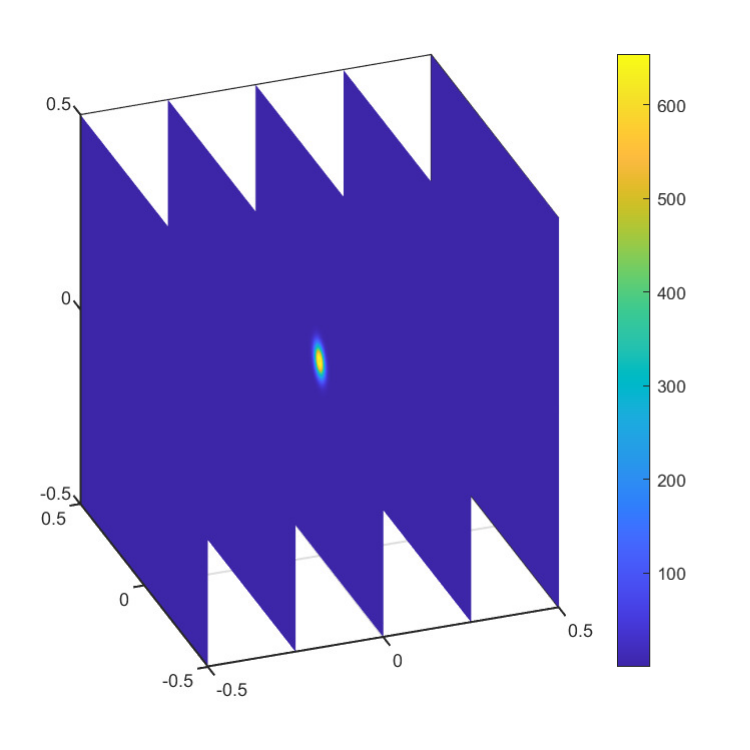}}\\
		\subfloat[Lymphocyte: $u$ at time $t=8\times10^{-6}$.]{\includegraphics[width=5.5cm, height=5cm]{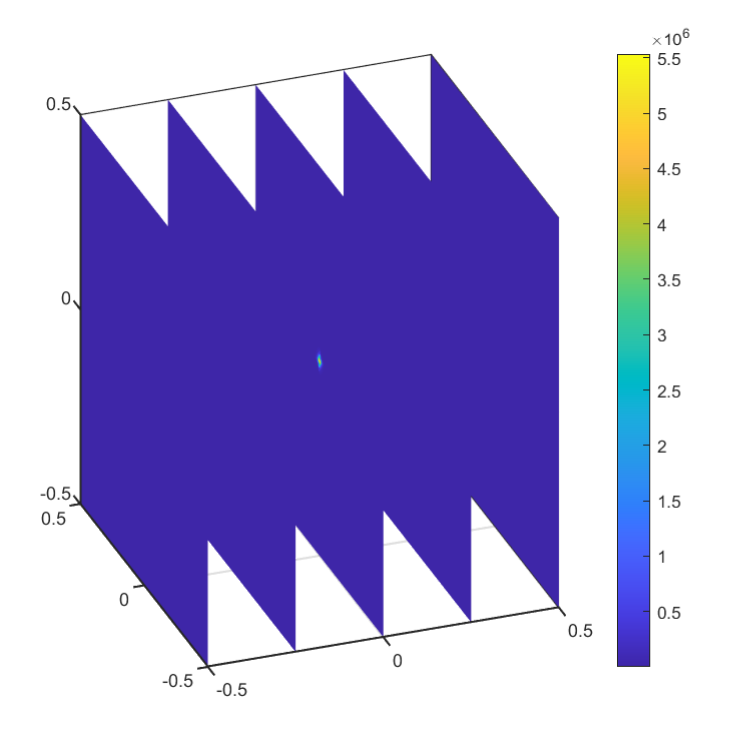}}
		\subfloat[Chemical concentration: $v$ at time $t=8\times10^{-6}$.]{\includegraphics[width=5.5cm, height=5cm]{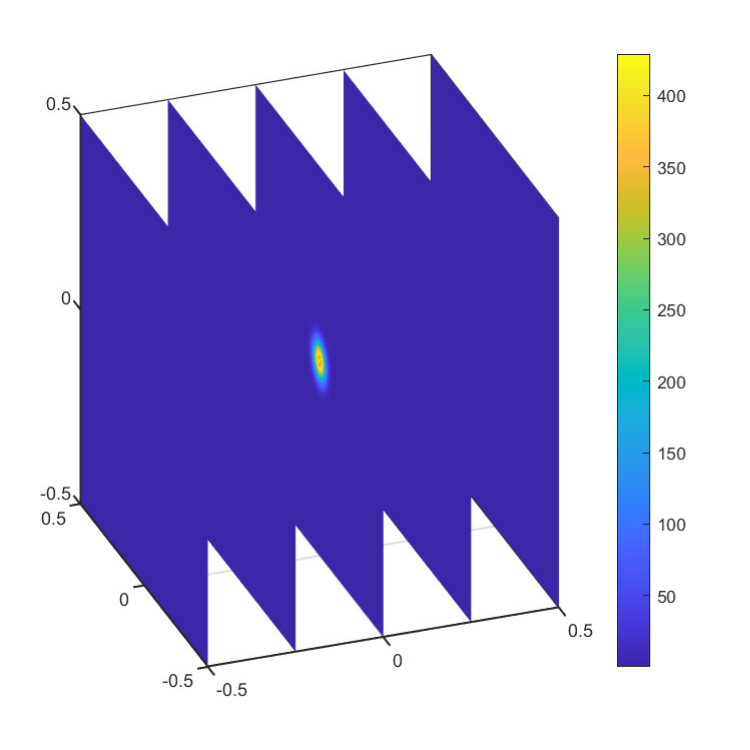}}
		\subfloat[Tumor cell: $w$ at time $t=8\times10^{-6}$.]{\includegraphics[width=5.5cm, height=5cm]{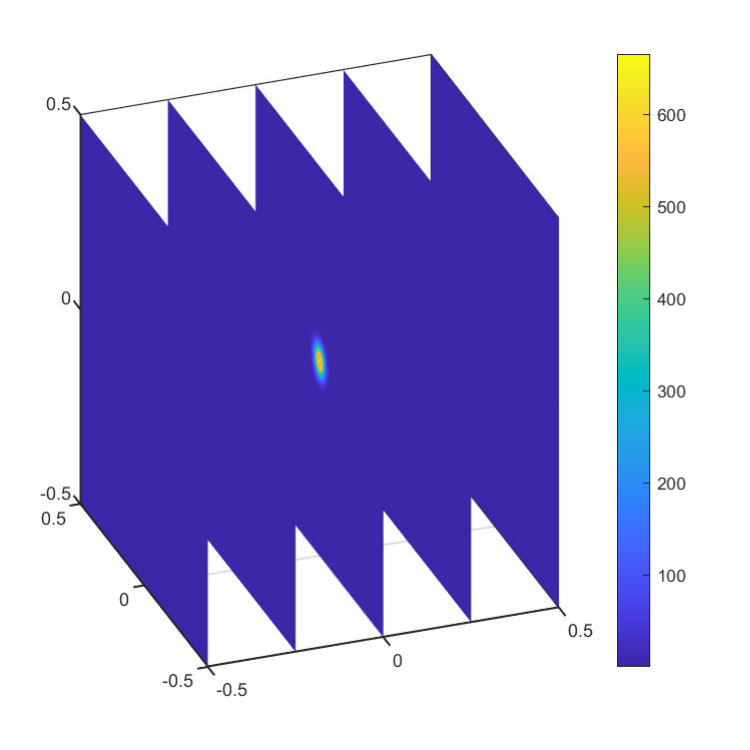}}
		\caption{Graphical representation of the evolution for $u,v,w$ at two different instants of time.}\label{fig.5.7.1}
	\end{figure}
	\def\w{.23\textwidth}
	\def\h{.21\textwidth}
	\begin{figure}[h]
		\centering
		\subfloat[Lymphocyte: $u$ for $t\in(0,\tmax)$.]{\label{u2}
			\begin{tikzpicture}
				\begin{axis}[%
					width=\w,height=\h,
					scale only axis,
					xmin=0, xmax=7e-06,
					xlabel={$t$},
					ymin=0, ymax=6000000,
					ylabel={$\Vert u(\cdot,t)\rVert_{L^\infty(\Omega)}$},
					]
					\addplot [color=mycolor1, line width=2.0pt, forget plot]
					table[row sep=crcr]{%
						0	3757.01044059917\\
						1.00024044513702e-06	14121.5786366928\\
						1.99954956769943e-06	51666.6227750583\\
						2.99979001283646e-06	172511.510269691\\
						4.00003045797348e-06	436894.140165122\\
						5.0002709031105e-06	269304.525383326\\
						5.99958002567291e-06	1282193.64323823\\
						6.99982047080994e-06	5524799.77536841\\
					};
				\end{axis}
		\end{tikzpicture}}
		\subfloat[Chemical concentration: $v$ for $t\in(0,\tmax)$.]{\label{v2}
			\begin{tikzpicture}
				\begin{axis}[%
					width=\w,height=\h,
					scale only axis,
					xmin=0,xmax=7e-06,
					xlabel={$t$},
					ymin=410,ymax=500,
					ylabel={$\Vert v(\cdot,t)\rVert_{L^\infty(\Omega)}$},
					]
					\addplot [color=mycolor1, line width=2.0pt, forget plot]
					table[row sep=crcr]{%
						0	498.068918645796\\
						9.99999997475243e-07	494.803076565744\\
						1.99999999495049e-06	486.559982770295\\
						2.99999999242573e-06	460.742323245732\\
						3.99999998990097e-06	432.521100586812\\
						4.99999998737621e-06	423.016137890729\\
						5.99999998485146e-06	410.16286238839\\
						6.9999999823267e-06	428.389469960257\\
					};
				\end{axis}
		\end{tikzpicture}}
		\subfloat[Tumor cell: $w$ for $t\in(0,\tmax)$.]{\label{w2} 
			\begin{tikzpicture}
				\begin{axis}[%
					width=\w,height=\h,
					scale only axis,
					xmin=0,xmax=7e-06,
					xlabel={$t$},
					ymin=600,ymax=800,
					ylabel={$\Vert w(\cdot,t)\rVert_{L^\infty(\Omega)}$},
					]
					\addplot [color=mycolor1, line width=2.0pt, forget plot]
					table[row sep=crcr]{%
						0	794.956217355833\\
						9.99999997475243e-07	787.814984536177\\
						1.99999999495049e-06	772.794091403289\\
						2.99999999242573e-06	729.946798026615\\
						3.99999998990097e-06	653.239798937889\\
						4.99999998737621e-06	629.588184533878\\
						5.99999998485146e-06	601.781701373807\\
						6.9999999823267e-06	664.503514499305\\
					};
				\end{axis}
		\end{tikzpicture}}
		\caption{Evolution of $\lVert u(\cdot,t)\rVert_{L^{\infty}(\Omega)},\lVert v(\cdot,t)\rVert_{L^{\infty}(\Omega)},\lVert w(\cdot,t)\rVert_{L^{\infty}(\Omega)}$ for $t\in(0,\tmax)$, $\tmax\approx8\times 10^{-6}$.}\label{fig.5.7.2}
	\end{figure}
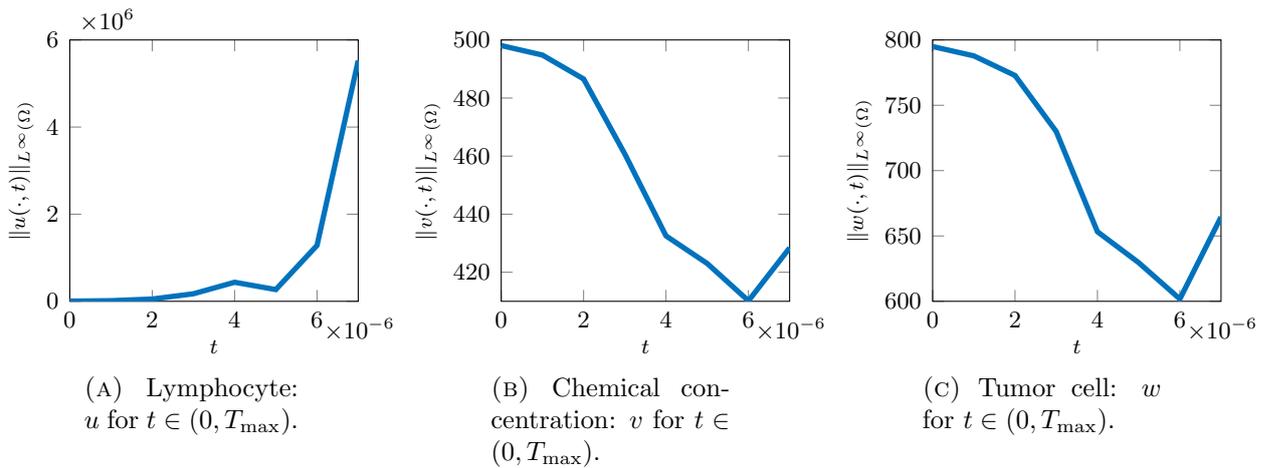
	
	\subsection*{Example of blow-up in 3D}
	Let us take in model \eqref{1}, $\tau=1$, $\Omega=[-0.5,0.5]\times[-0.5,0.5]\times[-0.5,0.5]$, $\chi=2, \alpha=\beta=\gamma=\delta=\mu=1$.
	Additionally we consider this radially symmetric bell-shaped initial data
	\begin{equation}\label{initialdata}
		\begin{split}
			u_0(x,y,z) = u(x,y,z,0)=&1000e^{-1000(x^2+y^2+z^2)},\\
			v_0(x,y,z)=v(x,y,z,0)=&500e^{-500(x^2+y^2+z^2)},\\
			w_0(x,y,z)=w(x,y,z,0)=&800e^{-800(x^2+y^2+z^2)}.
		\end{split}
	\end{equation}
	
	The grid for the cube is uniform and the subdivision is constructed with $\Delta x=\Delta y= \Delta z=\frac{1}{101}$ whereas $\Delta t=10^{-6}$. 
	
	The approximation of the solution $(u,v,w)$ is displayed. Figure~\ref{fig.5.7.1} collects the evolution of $u,v,w$ at two instants of time; conversely Figure~\ref{fig.5.7.2} shows the behaviour in time of the maximum of such unknowns in the cube $\Omega$. Precisely in this last figure, one can see a certain uncontrolled increase of $u$ when $t$ approaches $\tmax\approx 8\times10^{-6}$. In particular $\lVert u(\cdot,t)\rVert_{L^{\infty}(\Omega)}$ close to $\tmax$ is approximately $6\times10^{6}$ (see Figure \ref{u2}). On the other hand, and exactly in line with the theoretical bounds \eqref{l2.2} and \eqref{l2.3}, Figures \ref{w2} and \ref{v2} show that with the choice \eqref{initialdata} such bounds are respected and remain below
	\begin{equation*}
		M_2 = \max\left\{1, \|\bw_0\|_{\lis}\right\} = 800
		\quad \text{and} \quad
		M_3 = \max \left\{\frac{\alpha}{\beta}M_2, \|\bv_0\|_{\lis}\right\} = 800, 
	\end{equation*}
	respectively.
	
	This graphical representation enhances our understanding of the system's evolution and provides valuable insights into the dynamics leading up to the blow-up phenomenon.

	\section*{Acknowledgments}
	
	The second and third authors are members of the Gruppo Nazionale per l’Analisi Matematica, la Probabilità e le loro Applicazioni (GNAMPA) of the Istituto Nazionale di Alta Matematica (INdAM) and are partially supported by the research project {\em Analysis of PDEs in connection with real phenomena}, CUP F73C22001130007, funded by
	\href{https://www.fondazionedisardegna.it/}{Fondazione di Sardegna}, annuity 2021.	
	The second author acknowledges financial support by INdAM-GNAMPA project {\em problemi non lineari di tipo stazionario ed evolutivo}, CUP E53C23001670001.  The third author acknowledges financial support by PNRR e.INS Ecosystem of Innovation for Next Generation Sardinia (CUP F53C22000430001, codice MUR ECS0000038).

\end{document}